\documentclass[12pt]{amsart}
\usepackage[english]{babel}
\usepackage{amsfonts,amssymb,latexsym,amscd}

\theoremstyle{plain}
\newtheorem{theorem}{Theorem}

\newtheorem{proposition}{Proposition}
\newtheorem{corollary}{Corollary}
\theoremstyle{definition}
\newtheorem{definition}{Definition}
\theoremstyle{remark}

\newtheorem{example}{Example}
\newtheorem{question}{Question}

\begin{document}

\title{Groups of Binary Operations and Binary $G$-Spaces}
\author{Pavel S. Gevorgyan}

\address{Moscow State Pedagogical University}

\email{pgev@yandex.ru}

\begin{abstract}
The  group of continuous binary operations on a topological space is studied;
its relationship with the group of homeomorphisms is established.
The  category of binary $G$-spaces and bi-equivariant maps is constructed, which is a 
natural extension of the category of $G$-spaces and equivariant maps.  
Results related to the foundations of the theory of binary $G$-spaces are obtained.
\end{abstract}


\maketitle

\section{Auxiliary Results and Notation}

Throughout this paper, by a space we mean a topological space. All spaces are assumed to be 
Hausdorff.

We denote the category of topological spaces and continuous maps by $Top$.

By $C(X,Y)$ we  denote the space of all continuous maps of $X$ to $Y$ 
endowed with the compact-open topology, that is, the topology generated by the subbase consisting 
of all sets of the form $W(K, U)=\{f:X\to Y; \ f(K)\subset U\}$, 
where $K$ a compact subset of $X$ and $U$ is an open subset of $Y$.

All continuous function spaces are considered in the compact-open topology.

If $G$ is a topological group, then there is  a natural group operation on $C(X,G)$: 
given any continuous maps $f,g\in C(X,G)$, their product $fg\in C(X,G)$ is defined by 
$(fg)(x)=f(x)g(x)$ for all $x\in X$.

\begin{theorem}[\cite{McCoy}]\label{th_McCoy}
If $G$ is a topological group, then so is $C(X,G)$.
\end{theorem}

The group of all self-homeomorphisms of $X$ is denoted by $H(X)$. This group is not generally 
a topological group. However,  the following theorem is valid.

\begin{theorem}[\cite{Arens}]\label{th_arens}
If a space $X$ is locally compact and locally connected, then $H(X)$ is a topological group.
\end{theorem}

Let $X$ be a  topological space,   and let $G$ be a topological group $G$ with 
identity element $e$. Suppose given a continuous map $\theta:G \times X \to X$ 
satisfying the  conditions
$$ (1) \quad \theta(g, \theta(h, x))=\theta(gh,x) \qquad\text{and}\qquad 
(2) \quad \theta(e,x)=x $$
for $g, h \in G$ and $x\in X$. Then  $X$ is called a $G$-space, and  
the continuous map $\theta :G \times X \to X$ is called the action of the group $G$ 
on the $G$-space $X$. In this case, we use the notation $\theta(g,x)=gx$. 

Let $X$ and $Y$ be $G$-spaces. A continuous map $f: X\to Y$ is said to be equivariant 
if $f(gx)=gf(x)$ for any $g\in G$ and $x\in X$.

All $G$-spaces and their equivariant maps form a category. This category is denoted by $G$-$Top$.

The symmetric group on a set $X$ is denoted by $S(X)$. In the case of a finite set $X$, 
this group is denoted by $S_n(X)$ or $S_n$, where $n$ is the number of elements in $X$. 
The order of $S_n(X)$ equals $n!$: $|S_n(X)|=n!$

Details on these notions, as well as on all definitions, notions, and results used 
in this paper without reference, can be  found in \cite{Br}, \cite{Eng}, \cite{Gev1} and \cite{Gev}.

\section{The Group of Continuous Binary Operations}

Let $X$ be any topological space.

A continuous map $f:X^2\to X$ is called a \textit{continuous binary operation}, 
or a \textit{binary map}, on $X$.
We denote the set of all continuous binary operations on $X$ by $C_2(X)$. 
On $C_2(X)$ we define an operation ``$\ast$'' by 
\begin{equation}\label{eq4}
(f*\varphi) (t,x)=f(t, \varphi(t,x))
\end{equation}
for all $t,x\in X$.

\begin{proposition}
The set $C_2(X)$ under the operation ``$\ast$'' is a topological semigroup with identity element 
$e(t,x)=x$, i.e., a topological monoid.
\end{proposition}

\begin{proof}
Let $f$, $\varphi$, and $h$ be continuous binary operations on a topological space $X$. 
Let us check the semigroup axiom:
\begin{multline*}
[f*(\varphi*h)] (t,x)=f(t, (\varphi*h)(t,x))=f(t, \varphi(t, h(t,x)))\\
=(f*\varphi)(t, h(t,x))=[(f*\varphi)*h] (t,x).
\end{multline*}

The binary operation $e:X^2\to X$ defined by $e(t,x)=x$ is the identity element of the 
semigroup, because $f*e=e*f=f$. Indeed,
\[
(f*e)(t,x)=f(t, e(t,x))=f(t,x),
\]
\[
(e*f)(t, x)=e(t, f(t,x))=f(t,x).
\]
\end{proof}

\begin{definition}
A continuous binary operation $f\in C_2(X)$ is said to be \textit{invertible} if there exists 
a continuous binary operation $\varphi\in C_2(X)$ such that
\begin{equation}\label{eq6}
f*\varphi=\varphi*f=e.
\end{equation}
In this case, $f$ and $\varphi$ are said to be \textit{mutually inverse binary operations}.
\end{definition}

We denote  the set of all invertible elements of $C_2(X)$ by $H_2(X)$. The set $H_2(X)$ is a group.

\begin{example}
Let $X=\{a, b\}$ be a two-point discrete space. The symmetric group $S_2(X)$ of permutations of 
this set is the cyclic group $\mathbb{Z}_2$,  and the group of all invertible binary operations 
on $X=\{a, b\}$ is the group of order 4 with two generators 
$\varphi_1$ and $\varphi_2$, which are specified by 

\begin{center}
\begin{tabular}{|p{30pt}|p{30pt}|p{30pt}|p{30pt}|}
\hline
\raisebox{-2.80ex}[0cm][0cm]{$\varphi_1$:}&
$b$&
$a$&
 $a$ \\
\cline{2-4}&
$a$&
$b$&
$b$ \\
\cline{2-4}&
 &
a&
b \\
\hline
\end{tabular}
\end{center}

\begin{center}
\begin{tabular}{|p{30pt}|p{30pt}|p{30pt}|p{30pt}|}
\hline
\raisebox{-2.80ex}[0cm][0cm]{$\varphi_2$:}&
$b$&
$b$&
 $a$ \\
\cline{2-4}&
$a$&
$a$&
$b$ \\
\cline{2-4}&
 &
a&
b \\
\hline
\end{tabular}\vspace{2mm}
\end{center}
The multiplication table for this group is as follows:\vspace{2mm}

\begin{center}
\begin{tabular}{|p{30pt}|p{30pt}|p{30pt}|p{30pt}|}
\hline
$e$&
 $\varphi_1$&
 $\varphi_2$&
 $\varphi_3$ \\
\cline{1-4}
$\varphi_1$&
$e$&
$\varphi_3$ &
$\varphi_2$ \\
\cline{1-4}
$\varphi_2$&
 $\varphi_3$&
 $e$&
 $\varphi_1$ \\
\cline{1-4}
$\varphi_3$&
$\varphi_2$&
$\varphi_1$ &
$e$ \\
\hline
\end{tabular}\vspace{2mm}
\end{center}
This is the  Klein four-group.

In the case of a three-point set $X$, the order of $H_2(X)$ equals $(3!)^3=216$ 
(see Corollary~\ref{cor2} below).
\end{example}

\begin{theorem}\label{th01}
If a continuous binary operation $f:X^2\to X$ is invertible, then, for any $t\in X$, 
the continuous map $f_t:X\to X$ defined by 
\begin{equation}\label{eq5}
f_t(x)=f(t,x)
\end{equation}
for $x\in X$ is a homeomorphism.
\end{theorem}

\begin{proof}
Suppose that a binary operation $f:X^2\to X$ is invertible, i.e., there exists a binary operation  
$\varphi\in C_2(X)$ satisfying relations~\eqref{eq6}.
Take any element $t\in X$. Let us prove that the  continuous map $f_t:X\to X$ is a homeomorphism.

The map $f_t$ is a monomorphism. Indeed, suppose that $f_t(x)=f_t(x')$, i.e., $f(t,x)=f(t,x')$. 
Then
\begin{multline*}
    x=e(t,x)=(\varphi*f)(t,x)=\varphi(t,f(t,x))= \\
    =\varphi(t,f(t,x'))=(\varphi*f)(t,x')=e(t,x')=x'.
\end{multline*}

The continuous map $\varphi_t:X\to X$ defined by 
\begin{equation*}
    \varphi_t(x)=\varphi(t,x) 
\end{equation*}
is inverse to $f_t:X\to X$. Indeed,
\begin{multline*}
    (f_t\circ \varphi_t(x) = f_t(\varphi_t(x))=f_t(\varphi(t,x))=f(t,\varphi(t,x))=\\
    =(f*\varphi)(t,x)=e(t,x)=x,
\end{multline*}
 i.e., $f_t\circ \varphi_t=1_X$.  A similar argument proves that 
$\varphi_t\circ f_t=1_X$.
\end{proof}

It follows from Theorem~\ref{th01} that if an invertible binary operation $f:X^2\to X$ 
is represented in the form of a family
of homeomorphisms $\{f_t\}$, $t\in X$, then the inverse binary operation has the form 
$\{f^{-1}_t\}$.

The converse of Theorem~\ref{th01} is true for locally compact and locally connected spaces.

\begin{theorem}\label{th011}
Suppose that a space $X$ is locally compact and locally connected and $f:X^2\to X$ 
is a continuous binary operation. If the map $f_t:X\to X$ is a homeomorphism 
 for each $t\in X$, then the binary operation $f$ is invertible, and the inverse binary 
operation $f^{-1}$ is defined by $f^{-1}(t,x)=f_t^{-1}(x)$.
\end{theorem}

\begin{proof}
Consider the binary operation $\varphi$ given by $\varphi(t,x)=f_t^{-1}(x)$. Let us prove that 
this  binary operation is continuous and inverse to $f:X^2\to X$: $\varphi=f^{-1}$.

First, we prove the continuity of the map $\varphi:X^2\to X$. Let $(t_0,x_0)\in X^2$ be 
any point, and let $\varphi(t_0,x_0)=f_{t_0}^{-1}(x_0)=y_0$. Consider any open neighborhood 
$W\subset X$ of $y_0$ such that the closure $\overline{W}$ is compact. There exists a 
compact connected neighborhood  $K$ of $x_0$ for which 
\begin{equation}\label{KW}
f_{t_0}^{-1}(K)\subset W.
\end{equation}

We denote the interior of $K$ by $K^{\circ}$. We have 
\begin{equation}\label{xoinko}
f_{t_0}(y_0)=x_0\in K^{\circ}.
\end{equation}

Inclusion \eqref{KW}  implies 
\begin{equation}\label{dop}
f_{t_0}(W^C\cap \overline{W})\subset K^C,
\end{equation}
where $W^C$ and $K^C$ are the complements of $W$ and $K$, respectively.

Since $f:X^2\to X$ is a continuous binary operation,  $\{y_0\}$ 
and $W^C\cap \overline{W}$ are compact supsets of $X$, and $K^{\circ}$ and $K^C$ are open subsets, 
it follows from \eqref{xoinko} and \eqref{dop} that the point $t_0$ has an open neighborhood $U$ 
such  that, for any $t\in U$,  we have
\begin{equation}\label{xoinko1}
f_{t}(y_0)\in K^{\circ}
\end{equation}
and
\begin{equation}\label{dop1}
f_{t}(W^C\cap \overline{W})\subset K^C. 
\end{equation}

Inclusion \eqref{dop1} implies that 
\begin{equation*}\label{dop2}
K\subset f_{t}(W\cup \overline{W}^C)
\end{equation*}
for any $t\in U$. Therefore,
$$f_{t}^{-1}(K)\subset W\cup \overline{W}^C.$$
Since $f_{t}^{-1}(K)$ is connected and $W$ and $\overline{W}^C$ are disjoint open sets, 
 it follows from the last inclusion that $f_{t}^{-1}(K)$ is contained in one of the sets  
$W$ and $\overline{W}^C$. However, by virtue of \eqref{xoinko1}, we obviously have 
$f_{t}^{-1}(K)\subset W$. Hence 
\begin{equation}\label{ko}
f_{t}^{-1}(K^\circ)\subset W
\end{equation}
for all $t\in U$.

Thus, given any open neighborhood $W$ of $y_0=f_{t_0}^{-1}(x_0)$,  we have found open 
neighborhoods $U$ of $t_0$ and $K^\circ$ of $x_0$ for which \eqref{ko} holds. 
This proves the continuity of the binary operation $\varphi(t,x)=f_t^{-1}(x)$.

It is easy to show that the continuous binary operation $\varphi(t,x)=f_t^{-1}(x)$ 
is inverse to $f:X^2\to X$. Indeed,
$$(f*\varphi)(t,x)=f(t,\varphi(t,x))=f(t,f_t^{-1}(x))=f_t(f_t^{-1}(x))=x,$$
i.e. $f*\varphi=e$. The relation $\varphi*f=e$ is proved in a similar way.
\end{proof}

Theorems \ref{th01} and  \ref{th011} imply the following criterion for the invertibility 
of continuous binary operations on locally compact and locally connected spaces.

\begin{theorem}\label{th0111}
A continuous binary operation $f:X^2\to X$ 
on a  locally compact locally connected space $X$ is 
invertible if and only if 
the continuous map $f_t:X\to X$  defined by \eqref{eq5} is a homeomorphism for any $t\in X$.
\end{theorem}

\begin{proposition}
The group $H(X)$  of all self-homeomorphisms of a topological space $X$ is isomorphic 
(algebraically and topologically) to a subgroup of the group  $H_2(X)$ of invertible binary 
operations.
\end{proposition}

\begin{proof}
To each $f\in H(X)$ we assign the continuous map $\tilde{f}:X^2\to X$ defined by 
$\tilde{f}(t, x)=f(x)$, $t,x\in X$. Obviously, $\widetilde{f^{-1}}=\tilde{f}^{-1}$. Thus, 
$\tilde{f}$ is a continuous invertible binary operation, i.e., $\tilde{f}\in H_2(X)$. 
The correspondence $f \to \tilde{f}$ is the required isomorphism between the group $H(X)$ 
and a subgroup of $H_2(X)$.
\end{proof}

\begin{theorem}\label{th1}
For any  locally compact locally connected space $X$, the group $H_2(X)$ is 
isomorphic (algebraically and topologically) to $C(X,H(X))$. 
\end{theorem}

\begin{proof}
Consider the map $p:C(X,H(X))\to H_2(X)$ defined by 
\begin{equation}\label{p(f)}
p(f)(t,x)=f(t)(x) 
\end{equation}
for $f\in C(X,H(X))$ and  $t,x\in X$.  The map $f(t) : X \to X$ is a homeomorphism 
for each $t\in X$. Therefore, by virtue of Theorem~\ref{th0111}, the binary operation 
$p(f):X\times X\to X$ is invertible, that is, indeed belongs to the group $H_2(X)$.

Let us prove that $p$ is a monomorphism. Take $f,g\in C(X,H(X))$, $f\neq g$. There exists 
a point $t_0\in X$ such that $f(t_0)\neq g(t_0)$. Since $f(t_0), g(t_0) \in H(X)$, 
it follows that $f(t_0)(x_0)\neq g(t_0)(x_0)$ for some $x_0\in X$. Thus, 
$p(f)(t_0,x_0)\neq p(g)(t_0,x_0)$, and hence $p(f)\neq p(g)$.

The map $p$ is also an epimorphism. Indeed, let $\varphi \in H_2(X)$ be any continuous binary 
operation. By virtue of Theorem~\ref{th0111}, the map $\varphi_t:X\to X$ defined by 
$\varphi_t(x)=\varphi(t,x)$, $t, x\in X$, is a homeomorphism. It is easy to see 
that the element $f\in  C(X,H(X))$ determined by the condition $f(t)=\varphi_t$ is the preimage 
of the binary operation $\varphi$. Indeed, we have  $p(f)(t,x)=f(t)(x)=\varphi_t(x)=\varphi(t,x)$.

Thus, the map $p^{-1}:H_2(X)\to C(X,H(X))$ defined by 
$$p^{-1}(\varphi)(t)(x)= \varphi(t,x)$$
for $\varphi \in H_2(X)$ and $t,x\in X$ is inverse to  $p:C(X,H(X))\to H_2(X)$.

The map $p$ is a homomorphism, that is, $p(f\circ g)=p(f)*p(g)$. Indeed, 
for any  $t,x\in X$, we have 
\begin{multline*}
p(f\circ g)(t,x)=(f\circ g)(t)(x)= (f(t)\circ g(t))(x)=f(t)(g(t)(x))\\
=f(t)(p(g)(t,x))=p(f)(t,p(g)(t,x))=(p(f)*p(g))(t,x).
\end{multline*}

Let us prove the continuity of $p$. Take any element $W(K\times K', U)$ of the subbase 
of the compact-open topology on $H_2(X)$, where $U\subset X$ open and $K,K'\subset X$ are compact 
subsets of $X$. Let us show that the preimage of $W(K\times K', U)$ is the set $W(K, W(K', U))$, 
which is an element of the subbase of the compact-open topology on $C(X,H(X))$. 
Indeed, for any  $\varphi \in W(K\times K', U)$ and $f=p^{-1}(\varphi)\in  C(X,H(X))$, we have 
\begin{multline*}
\varphi \in W(K\times K', U) \iff \varphi(t,x)\in U \iff p(f)(t,x)\in U \\
\iff f(t)(x)\in U \iff f\in W(K, W(K', U)),
\end{multline*}
where $t\in K$ and $x\in K'$ are arbitrary elements.

The continuity of the inverse map $p^{-1}:H_2(X)\to C(X,H(X))$ is proved in precisely the same way.
\end{proof}

\begin{corollary}
If $X$ is a locally compact locally connected space, then $H_2(X)$ is a topological group.
\end{corollary}

\begin{proof}
By Theorem~\ref{th_arens}, $H(X)$ is a topological group. Thus,  $C(X,H(X))$ is 
a topological group as well (by Theorem~\ref{th_McCoy}). According to Theorem~\ref{th1}, 
$H_2(X)$ is a topological group.
\end{proof}

\begin{corollary}\label{cor2}
If $|X|=n<\infty$, then $|H_2(X)|=(n!)^n$.
\end{corollary}

\begin{proof}
For a finite set $X$, we have $H(X)=S_n(X)$, where $S_n(X)$ is the symmetric group of 
permutations of $X$.  Since $|S_n(X)|=n!$, the corollary follows directly from Theorem~\ref{th1}.
\end{proof}


\section{Binary Actions of Groups}

\begin{definition}\label{def1}
Let $G$ be a topological group, and let $X$ be a space. 
A continuous map $\alpha :G\times X^2\to X$ is called a \textit{binary action} 
of $G$ on $X$ if   the following conditions hold: 
\begin{equation}\label{eq(1)}
\alpha (gh, t,x)=\alpha(g, t, \alpha(h,t,x)),
\end{equation}
\begin{equation}\label{eq(2)}
\alpha (e, t,x)=x,
\end{equation}
where $e$ is  the identity element of $G$, $g,h\in G$, and $t,x\in X$.

We refer to a space $X$ together with a fixed binary action of a group $G$, that is, to a triple 
$(G,X,\alpha)$, as a \textit{binary $G$-space}.
\end{definition}

Note that in the notation
\begin{equation}\label{eq00}
g(t,x)=\alpha (g, t,x)
\end{equation}
relations \eqref{eq(1)} and \eqref{eq(2)} take the form 
$$gh(t,x)=g(t, h(t,x)), \quad e(t,x)=x.$$

\begin{example}\label{pr1}
The continuous map $\alpha :G\times G^2\to G$ defined by 
\begin{equation}\label{eq2}
\alpha(g,h_1,h_2)=h_1gh_1^{-1}h_2, \quad \text{or} \quad g(h_1,h_2)=h_1gh_1^{-1}h_2,
\end{equation}
is a binary action of the topological group $G$ on itself.

Indeed, conditions \eqref{eq(1)} and \eqref{eq(2)} in Definition~\ref{def1} are satisfied: we have 
\begin{multline*}
    \alpha(gh,h_1,h_2)=h_1ghh_1^{-1}h_2=h_1gh_1^{-1}h_1hh_1^{-1}h_2=\\
    =\alpha(g, h_1, h_1hh_1^{-1}h_2)=\alpha(g, h_1, \alpha(h,h_1,h_2))
\end{multline*}
and
$$\alpha (e, h_1,h_2)=h_1eh_1^{-1}h_2=h_2$$
for all $g,h,h_1,h_2\in G$.
\end{example}

\begin{example}\label{pr2}
Let $GL(n, \mathbf{R})$ be the general linear group of degree $n$. Consider the continuous map
$\alpha:GL(n, \mathbf{R})\times \mathbf{R}^n\times \mathbf{R}^n \to \mathbf{R}^n$ defined by 
\begin{equation}\label{eq7-1}
\alpha (A,\mathbf{x},\mathbf{y})=(E-A)\mathbf{x}+A\mathbf{y},
\end{equation}
or, equivalently, 
\begin{equation}\label{eq7}
A(\mathbf{x},\mathbf{y})=(E-A)\mathbf{x}+A\mathbf{y},
\end{equation}
where $A\in GL(n, \mathbf{R})$, $E$ is the identity matrix, 
and $\mathbf{x},\mathbf{y} \in \mathbf{R}^n$.

Note that, for any $A, B \in GL(n, \mathbf{R})$, we have 
\begin{multline*}
A(\mathbf{x},B(\mathbf{x},\mathbf{y}))=(E-A)\mathbf{x}+A(B(\mathbf{x},\mathbf{y}))=
(E-A)\mathbf{x}+A((E-B)\mathbf{x}+B\mathbf{y})\\
=(E-A)\mathbf{x}+A(E-B)\mathbf{x}+AB\mathbf{y}=(E-A+A-AB)\mathbf{x}+AB\mathbf{y}\\
=(E-AB)\mathbf{x}+AB\mathbf{y}=AB(\mathbf{x},\mathbf{y})
\end{multline*}
and
$$E(\mathbf{x},\mathbf{y})=(E-E)\mathbf{x}+E\mathbf{y}=\mathbf{y}.$$
Therefore, relation~\eqref{eq7} defines a binary action of the general linear group 
$GL(n, \mathbf{R})$ on
the $n$-dimensional vector space $\mathbf{R}^n$.
\end{example}

Let $\alpha$ be a binary action of a topological group $G$ on a space $X$. 
For each $g\in G$, we define a continuous map $\alpha_g:X^2\to X$ as 
\begin{equation}\label{eq3}
\alpha_g(t,x)= \alpha(g,t,x).
\end{equation}
Considitions \eqref{eq(1)} and \eqref{eq(2)} imply that 
$\alpha_{gh}=\alpha_g * \alpha_h$ and $\alpha_e$ is the identity element in the monoid of 
binary operations on $X$.
Thus,  the following proposition is valid.

\begin{proposition}\label{pr01}
The map $g\to \alpha_g$ is a continuous homomorphism from $G$ to the group $H_2(X)$ of all invertible 
continuous binary operations.
\end{proposition}

Thus, the elements of the group $G$ can be treated as continuous binary operations on the 
space $X$ which act by the rule $g(t,x)=\alpha_g(t,x)= \alpha(g,t,x)$.

Take any $t\in X$. Consider the continuous map $\alpha_t:G\times X\to X$
defined by 
\begin{equation}\label{eq1}
\alpha_t(g,x)= \alpha(g,t,x).
\end{equation}

\begin{proposition}
The map $\alpha_t$ is an action of the group $G$ on the space $X$.
\end{proposition}

\begin{proof}
The map $\alpha_t$ satisfies the conditions in the definition of an action 
of a group on a topological space. Indeed, we have 

(1) $\alpha_t(gh,x)=\alpha(gh,t,x)=\alpha(g, t, \alpha(h,t,x))=
\alpha(g, t, \alpha_t(h,x))=\alpha_t(g, \alpha_t(h,x))$ and 

(2) $\alpha_t(e,x)=\alpha(e,t,x)=x$

\noindent
for all $g,h\in G$ and $x\in X$.
\end{proof}

Thus, \textit{a binary action $\alpha$ of a group $G$ on a space $X$ induces the family  
$\{\alpha_t\}$, $t\in X$, of 
``ordinary'' actions of $G$ on $X$}.

In the case of the binary action \eqref{eq2} (see Example~\ref{pr1}), all induced 
actions of $G$ on itself are equivalent. To be more precise,
the homeomorphism $f:G\to G$ given by $f(g)=h^{-1}gh$, where $h$ is any fixed element of $G$, 
is an equivalence with respect to the actions $\alpha_e$ and $\alpha_h$. Indeed,
\begin{multline*}               
f(\alpha_e(g,\tilde{g}))=f(e^{-1}ge\tilde{g})=f(g\tilde{g})=h^{-1}g\tilde{g}=\\
    =h^{-1}ghh^{-1}\tilde{g}h=h^{-1}ghf(\tilde{g})=\alpha_h(g,f(\tilde{g})).
\end{multline*}

The binary action of the general linear group $GL(n, \mathbf{R})$ on
$\mathbf{R}^n$ defined by \eqref{eq7-1} (see Example~\ref{pr2}) 
induces the family $\{\alpha_\textbf{a}; \textbf{a}\in \mathbf{R}^n\}$ of ordinary actions 
of the group $GL(n, \mathbf{R})$ on $\mathbf{R}^n$.
It is easy to show that 
all these actions are equivalent. To be more precise, 
the map $f(\textbf{x})=\textbf{x}-\textbf{a}$ is an equivariant self-homeomorphism
of the space $\mathbf{R}^n$ with respect to the actions $\alpha_\textbf{a}$ 
and $\alpha_\textbf{0}$.  Moreover, the action $\alpha_{\textbf{0}}$ is 
multiplication of a given matrix by a given element of $\mathbf{R}^n$: 
$\alpha_{\textbf{0}}(A,\textbf{y})=\alpha(A,\textbf{0},\textbf{y})=\textbf{0}+A\textbf{y}=
A\textbf{y}$.

Given $A\subset X$ and $g\in G$, we set $gA=\{g(a,a'); a,a'\in A\}$. Similarly, 
given $K\subset G$, we set
$KA=\{g(a,a'); g\in K, a,a'\in A\}$.

\begin{theorem}
Let $A$ be a subset of a binary $G$-space $X$, and let $K$ be a subset of a 
compact topological group $G$. Then the following assertions hold: 

(i) if $A$ is open, then so is $KA$; 

(ii) if $A$ is compact and $K$ is closed, then $KA$ is compact. 
\end{theorem}

\begin{proof}
(i) Note that $KA=\bigcup\limits_{g\in K}\bigcup\limits_{a\in A}g_a(A)$, where $g_a:X\to X$ 
is the map defined by $g_a(x)=g(a,x)$, $x\in X$.  By virtue of Theorem~\ref{th01},   
$g_a:X\to X$ is a homeomorphism; therefore, the set $g_a(A)$ is open. Hence $KA$ is open as well.

(ii) If $K$ is closed, then it is also compact, because $G$ is a compact group. 
Thus, the set $K\times A^2$ is compact, and so is its continuous image $KA=\alpha(K\times A^2)$. 
\end{proof}

\begin{corollary}
Suppose that $G$ is a compact group, $X$ is a binary $G$-space, and $A\subset X$ is a 
subset of $X$.  If $A$ is open (compact), then so is $GA$.
\end{corollary}


\section{The Category of Binary $G$-Spaces}

Let $(G,X,\alpha)$ and $(G,Y,\beta)$ be two binary $G$-spaces.
\begin{definition}
We say that a continuous map $f:X\to Y$ is \textit{bi-equivariant} if 
$$f(\alpha(g,t,x))=\beta(g,f(t),f(x)),$$
or, equivalently,
$$f(g(t,x))=g(f(t),f(x)),$$
for all $g\in G$ and $t,x \in X$.
\end{definition}

We refer to a bi-equivariant map $f:X\to Y$ which is simultaneously a homeomorphism as 
a \textit{bi-equivalence} of binary $G$-spaces.
Note that the inverse map $f^{-1}:Y\to X$ is bi-equivariant as well. Indeed, we have 
\begin{multline*}
    f^{-1}(g(y,y'))=f^{-1}(g(f(x),f(x'))=f^{-1}(f(g(x,x')))=\\
    =g(x,x')=g(f^{-1}(y),f^{-1}(y')),
\end{multline*}
where $x,x'\in X$ and  $y,y'\in Y$.

The following assertion is valid.

\begin{theorem}
The binary $G$-spaces and bi-equivariant maps form a category.
\end{theorem}

We denote the category of binary $G$-spaces and bi-equivariant maps by $G$-$Top^2$.

\begin{proposition}
Any bi-equivariant map $f:X\to Y$ is equivariant with respect to the induced actions $\alpha_t$ 
and $\beta_{f(t)}$ for any $t\in X$.
\end{proposition}

\begin{proof}
Indeed, 
$$f(\alpha_t(g,x))=f(\alpha(g,t,x))=\beta(g,f(t),f(x)) = 
\beta_{f(t)}(g,f(x)).$$
\end{proof}

Now, let $X$ be any $G$-space. On $X$ we define a binary $G$-action by 
\begin{equation}\label{eq11}
g(x,x')=g\cdot x'
\end{equation}
for all $g\in G$ and $x, x'\in X$.

Note that if $X$ and $Y$ are $G$-spaces, then any equivariant map $f:X\to Y$
is bi-equivariant with respect to the action~\eqref{eq11}. Indeed,
$$f(g(x,x')=f(g\cdot x')=g\cdot f(x')=g(f(x), f(x')).$$

Thus, the category $G$-$Top$ is a subcategory of the category $G$-$Top^2$. 
We have the following chain of natural extensions of categories:
$$Top\subset G\text{-}Top\subset G\text{-}Top^2.$$

Let $(G, X, \alpha)$ be a binary $G$-space. Consider the $G$-space 
\mbox{$(G, X\times X, \tilde{\alpha})$} on which the group $G$ acts as 
\begin{equation*}\label{eq10}
\tilde{\alpha}(g, x,x')=(x, \alpha(g,x,x')).
\end{equation*}
Using the notations $\tilde{\alpha}(g, x,x')=g\cdot (x,x')$ and $\alpha(g,x,x')=g(x,x')$, we 
rewrite this formula as 
$$g\cdot (x,x')=(x, g(x,x')).$$

Note that if $X$ and $Y$ are binary $G$-spaces, then any bi-equivariant map $f:X\to Y$
generates the equivariant map $\tilde{f}:X\times X\to Y\times Y$  defined by 
\begin{equation*}\label{eq9}
\tilde{f}(x,x')=(f(x), f(x')),
\end{equation*}
where $x,x'\in X$. Indeed,
\begin{multline*}
\tilde{f}(g\cdot(x,x'))=\tilde{f}(x, g(x,x'))=(f(x), f(g(x,x')))\\
=(f(x), g(f(x),f(x')))=g\cdot(f(x),f(x'))= g\cdot\tilde{f}(x, x').
\end{multline*}

\textit{This correspondence is a covariant functor from the category $G$-$TOP^2$ 
to the category $G$-$TOP$}.

\section{Invariant Sets}

\begin{definition}
We say that a subset $A\subset X$ is \textit{invariant} with respect to the binary action of a 
group $G$ if $GA=A$.
\end{definition}

It is easy to see that the intersection $A\cap B$ of invariant  sets $A, B \subset X$ 
is invariant.
However, the union $A\cup B$ of two invariant sets is not  generally invariant. 
Indeed, any one-point set $\{x\}$, $x\in \mathbf{R}^n$, is invariant with respect to 
the binary action \eqref{eq7-1}  of the general linear group $GL(n, \mathbf{R})$ on
the $n$-dimensional vector space $\mathbf{R}^n$ (see Example~\ref{pr2}), since  
$A(x,x)=(E-A)x+Ax=x$ for all $A\in GL(n, \mathbf{R})$. However, the union 
$\{x\}\cup \{y\}$ of two one-point sets is not invariant, because 
$GL(n, \mathbf{R})\{x,y\}=\mathbf{R}^n$.

\begin{question}
What are binary $G$-spaces in which the union $A\cup B$ of any invariant subsets $A, B \subset X$ 
is invariant?
\end{question}

\begin{definition}
The orbit of an element $x\in X$ of a binary $G$-space $X$ is defined as the  
minimal invariant set $[x]\subset X$ containing $x$.
\end{definition}

Obviously,  $x\in G(x,x)\subset [x]$ for all $x\in X$. Therefore, if $G(x,x)$ is 
invariant, then $G(x,x)= [x]$. The following problem naturally arises.
\begin{question}
In what binary $G$-spaces is the set $G(x,x)$ invariant?
\end{question}

The solution of this problem in a special case is given by Theorem~\ref{th100} below.

\begin{definition}
A binary $G$-space $X$ is said to be \textit{distributive} if 
\begin{equation}\label{eq1-1}
g(h(x,x'), h(x,x''))=h(x,g(x', x''))
\end{equation}
for all $x, x', x'' \in X$ and $g, h\in G$.
\end{definition}

\begin{example}
A group $G$ with binary action \eqref{eq2} (see Example~\ref{pr1}) is a distributive binary 
$G$-space. Indeed, for any $g,h,k,k_1,k_2\in G$, we have 
\begin{multline*}
g(h(k,k_1), h(k,k_2))=g(khk^{-1}k_1, khk^{-1}k_2)= \\
=khk^{-1}k_1gk_1^{-1}kh^{-1}k^{-1}khk^{-1}k_2=khk^{-1}k_1gk_1^{-1}k_2= \\
=h(k,k_1gk_1^{-1}k_2)=h(k,g(k_1, k_2)).
\end{multline*}
\end{example}

\begin{theorem}\label{th100}
If $X$ is a distributive binary $G$-space, then the set $G(x,x)$ is invariant for any $x\in X$.
\end{theorem}

\begin{proof}
We must show that  $g(h(x,x), k(x,x))\in G(x,x)$ 
for any $h(x,x)$, $k(x,x)\in G(x,x)$ and any $g\in G$. By virtue of~\eqref{eq1-1}, we have 
\begin{multline*}
g(h(x,x), k(x,x))=g(h(x,x), h(x,h^{-1}k(x,x)))=\\
=h(x, g(x, h^{-1}k(x,x))) = h(x, gh^{-1}k(x, x))=\\
=hgh^{-1}k(x, x)\in G(x,x),
\end{multline*}
because $hgh^{-1}k\in G$.
\end{proof}

\begin{question}
Is the converse of Theorem~\ref{th100} true?
\end{question}

\begin{question}
Given a  distributive $G$-space $X$ and $x'\notin G(x,x)$, is the  set $G(x,x')$ invariant?
\end{question}

\bibliographystyle{plain}
\bibliography{ArXiv_Groups_of_binary_operations_and_binary_G-spaces}

\end{document}